\documentclass{scrartcl}
\usepackage[T1]{fontenc}
\usepackage[utf8]{inputenc}
\usepackage{amsmath, amsthm, amssymb}
\usepackage{dsfont}
\usepackage{graphicx}
\usepackage[format=plain]{caption}
\usepackage{capt-of}
\usepackage{subcaption}
\usepackage{bm}

\newcommand{\norm}[1]{\left\| #1 \right\|}  
\renewcommand{\d}{\,\mathrm{d}} 
\newcommand{\e}{\mathrm{e}} 
\newcommand{\ddt}[1]{\frac{\mathrm{d}} {\mathrm{d}t} #1}   
\newcommand{\N}{\mathbb{N}}  

\newcommand{\R}{\mathbb{R}}

\newcommand{\dist}{\operatorname{dist}}

\newcommand{\eps}{\varepsilon}
\renewcommand{\phi}{\varphi}
\newcommand{\ul}{\underline}
\newcommand{\ol}{\overline}

\numberwithin{equation}{section}

\newtheorem{thm}{Theorem}[section]

\newtheorem{lm}[thm]{Lemma}
\newtheorem{cond}[thm]{Condition}

\theoremstyle{definition} 
\theoremstyle{definition}

\title{A local input-to-state stability result w.r.t.~attractors of nonlinear reaction-diffusion equations} 

\author{Sergey Dashkovskiy$^{1}$, Oleksiy Kapustyan$^{2}$, Jochen Schmid$^{1,3}$\\  
\small $^1$ Institute for Mathematics, University of W\"urzburg, 97074 W\"urzburg, Germany\\
\small $^2$ Kiev National Taras Shevchenko University, 01033 Kiev, Ukraine\\
\small $^3$ Fraunhofer Institute for Industrial Mathematics (ITWM), 67663 Kaiserslautern, Germany\\ 
\small sergey.dashkovskiy@mathematik.uni-wuerzburg.de, alexkap@univ.kiev.ua,\\ \small jochen.schmid@itwm.fraunhofer.de}  

\date{}

\begin{document}

\maketitle

\begin{abstract}
\small{ \noindent 
We establish the local input-to-state stability of a large class of disturbed nonlinear reaction-diffusion equations w.r.t.~the global attractor of the respective undisturbed system. 
}
\end{abstract}


\section{Introduction}

In this paper, we are concerned with disturbed nonlinear reaction-diffusion 
equations of the form
\begin{align} \label{eq:react-diffus-eq}
\begin{split}
\partial_t y(t,\zeta) &= \Delta y(t,\zeta) + g(y(t,\zeta)) + h(\zeta) u(t) \qquad (\zeta \in \Omega) \\
y(t,\zeta) &= 0 \qquad (\zeta \in \partial \Omega)
\end{split} 
\end{align}
on a bounded domain $\Omega \subset \R^d$ 
with smooth boundary $\partial \Omega$, 
where $g \in C^1(\R,\R)$ and $h \in L^2(\Omega,\R)$ and the disturbance $u$ belongs to $\mathcal{U} := L^{\infty}([0,\infty),\R)$. It is well-known~\cite{Robinson} that the corresponding undisturbed equation
\begin{align} \label{eq:react-diffus-eq, undisturbed}
\begin{split}
\partial_t y(t,\zeta) &= \Delta y(t,\zeta) + g(y(t,\zeta)) \qquad (\zeta \in \Omega) \\
y(t,\zeta) &= 0 \qquad (\zeta \in \partial \Omega)
\end{split} 
\end{align}
has a unique global attractor $\Theta \subset X := L^2(\Omega,\R)$ under suitable growth and upper-boundedness conditions on the nonlinearity $g$ and its derivative $g'$ respectively. As usual, a global attractor for~\eqref{eq:react-diffus-eq, undisturbed} is defined to be a compact subset of $X$ that is invariant and uniformly attractive for~\eqref{eq:react-diffus-eq, undisturbed}. Also, it can be shown~\cite{KaVa09} that the global attractor $\Theta$ of~\eqref{eq:react-diffus-eq, undisturbed} is a stable set for~\eqref{eq:react-diffus-eq, undisturbed}. 
\smallskip

What we show in this paper is that the disturbed reaction-diffusion equations~\eqref{eq:react-diffus-eq} are locally input-to-state stable w.r.t.~the global attractor $\Theta$ of the undisturbed equation~\eqref{eq:react-diffus-eq, undisturbed}. 
So, 
we show that there exist comparison functions $\beta \in \mathcal{KL}$ and $\gamma \in \mathcal{K}$ and radii $r_{0x}, r_{0u} > 0$ such that for every initial value $y_0 \in X$ with $\norm{y_0}_{\Theta} \le r_{0x}$ and every disturbance $u \in \mathcal{U}$ with $\norm{u}_{\infty} \le r_{0u}$ the 
global weak solution $$[0,\infty) \ni t \mapsto y(t,\cdot) = y(t,y_0,u) \in X$$ 
of the boundary value problem~\eqref{eq:react-diffus-eq}  
with initial condition $y(0,\cdot) = y_0 \in X$ satisfies the following estimate:
\begin{align} \label{eq:lISS-estimate,intro}
\norm{y(t,y_0,u)}_{\Theta} \le \beta(\norm{y_0}_{\Theta},t) + \gamma(\norm{u}_{\infty}) \qquad (t \in [0,\infty)). 
\end{align}
See~\cite{Mi16} for the analogous definition in the special case~$\Theta = \{0\}$. 
%
In the above relations, we use the standard notation
\begin{align} \label{eq:norm-Theta,def}
\norm{x}_{\Theta} := \dist(x,\Theta) := \inf_{\theta \in \Theta} \norm{x-\theta} \qquad (x \in X)
\end{align}
and the standard definitions for the comparison function classes $\mathcal{KL}$ and $\mathcal{K}$, which are recalled in~\eqref{eq:comparison-fct-classes} below. 
In words, 
the local input-to-state stability estimate~\eqref{eq:lISS-estimate,intro} means 
that  
\begin{itemize}
\item[(i)] the 
invariant set $\Theta$ for~\eqref{eq:react-diffus-eq, undisturbed} is locally stable and 
attractive for the undisturbed system~\eqref{eq:react-diffus-eq, undisturbed} and 
\item[(ii)] these local stability and attractivity properties are affected only slightly in the presence of disturbances of small magnitude $\norm{u}_{\infty}$. 
\end{itemize}
In order to achieve the estimate~\eqref{eq:lISS-estimate,intro}, we will construct a suitable local input-to-state Lyapunov function $V$.
\smallskip

As far as we know, our result is the first (local) input-to-state stability result w.r.t.~attractors $\Theta$ of concrete partial differential equation systems. All previous concrete pde results 
we are aware of -- like those from~\cite{DaMi13}, \cite{JaNaPaSc16}, \cite{JaSc18}, \cite{KaKr16}, \cite{KaKr17}, \cite{MaPr11}, \cite{MiKaKr17}, \cite{ScZw18}, \cite{TaPrTa17}, \cite{ZhZh17a}, \cite{ZhZh17b}, for instance -- 
establish input-to-state stability only w.r.t.~an equilibrium point $\theta$, which without loss of generality is assumed to be $\theta = 0$. In particular, all these previous results require their   nonlinearity $g$ to be such that $g(\theta) = g(0) = 0$ and such that the undisturbed system has the singleton $\Theta := \{\theta\} = \{0\}$ as an attractor.  
With our result, by contrast, we can treat much more general nonlinearities: 
we can treat nonlinearities $g$ with $g(0) \ne 0$ and, more importantly, nonlinearities $g$ for which the undisturbed system~\eqref{eq:react-diffus-eq, undisturbed} has only a non-singleton attractor $\Theta \supsetneq \{0\}$. A simple example of such a nonlinearity is given by $g(r) := -r^3 + r$, which leads to the Chaffee--Infante equation. 
We refer to~\cite{KaKaVa15}, \cite{GoKaKaPa14}, \cite{GoKakaKh15} \cite{DaKaRo17} for other interesting results about non-trivial global attractors of nonlinear, impulsive, or even multi-valued semigroups. 
\smallskip

In the entire paper, we will use the following conventions and notations. As above, $X := L^2(\Omega,\R)$ and $\mathcal{U} := L^{\infty}(\R^+_0,\R)$ with $\R^+_0 := [0,\infty)$ and with the standard norm of $X$ being denoted simply by $\norm{\cdot} := \norm{\cdot}_{L^2(\Omega)}$. As usual, 
\begin{align*}
B_r(x_0) = B_r^{X}(x_0), \quad \ol{B}_r(x_0) = \ol{B}_r^{X}(x_0) \quad \text{and} \quad B_r(u_0) = B_r^{\mathcal{U}}(u_0), \quad  \ol{B}_r(u_0) = \ol{B}_r^{\mathcal{U}}(u_0)
\end{align*}
denote the open and closed balls in $X$ or $\mathcal{U}$ of radius $r$ around $x_0 \in X$ or $u_0 \in \mathcal{U}$ respectively. We will often use the notation~\eqref{eq:norm-Theta,def} and
\begin{align*}
B_r(\Theta) := \{x \in X: \norm{x}_{\Theta} < r \} 
\qquad \text{and} \qquad
\ol{B}_r(\Theta) := \{x \in X: \norm{x}_{\Theta} \le r \},
\end{align*}
as well as the notation 
$\dist(M,\Theta) := \sup_{x\in M} \norm{x}_{\Theta}$
for subsets $M, \Theta \subset X$. Also, $\mathcal{K}$, $\mathcal{K}_{\infty}$ and $\mathcal{KL}$ will denote the following standard classes of comparison functions:
\begin{gather}
\mathcal{K} := \{ \gamma \in C(\R^+_0,\R^+_0): \gamma \text{ strictly increasing with } \gamma(0) = 0 \} \notag \\
\mathcal{K}_{\infty} := \{ \gamma \in \mathcal{K}: \gamma \text{ unbounded} \}  \label{eq:comparison-fct-classes}\\
\mathcal{KL} := \{ \beta \in C(\R^+_0 \times \R^+_0,\R^+_0): \beta(\cdot,t) \in \mathcal{K} \text{ for } t \ge 0 \text{ and } \beta(s,\cdot) \in \mathcal{L} \text{ for } s > 0 \}, \notag
\end{gather}
where $\mathcal{L} := \{ \gamma \in C(\R^+_0,\R^+_0): \gamma \text{ strictly decreasing with } \lim_{t\to\infty} \gamma(t) = 0 \}$. And finally, upper right-hand Dini derivatives will be denoted by
\begin{align*}
\ol{\partial}_t^+ v(t) := \varlimsup_{\tau \to 0+} \frac{v(t+\tau)-v(t)}{\tau}.
\end{align*}

\section{Some preliminaries}

In this section, we provide 
the necessary preliminaries for our local input-to-state stability result. 
We begin by recalling the definition of weak solutions of initial boundary value problems of the form 
\begin{align} \label{eq:ibvp}
\begin{split}
\partial_t y(t,\zeta) &= \Delta y(t,\zeta) + g(y(t,\zeta)) + h(\zeta) u(t) \qquad ((t,\zeta) \in [s,\infty) \times \Omega) \\
y(t,\cdot)|_{\partial \Omega} &= 0 \qquad \text{and} \qquad y(s,\cdot) = y_s \qquad (t \in [s,\infty)).
\end{split} 
\end{align}
In fact, we will have to consider initial boundary value problems with more general inhomogeneities of the form 
\begin{align} \label{eq:ibvp-general}
\begin{split}
\partial_t y(t,\zeta) &= \Delta y(t,\zeta) + \ol{g}(y(t,\zeta)) + \ol{h}(t,\zeta)  \qquad ((t,\zeta) \in [s,\infty) \times \Omega) \\
y(t,\cdot)|_{\partial \Omega} &= 0 \qquad \text{and} \qquad y(s,\cdot) = y_s \qquad (t \in [s,\infty)),
\end{split} 
\end{align}
where $\ol{g}$, $\ol{h}$ satisfy the following conditions.

\begin{cond} \label{cond:ol-g,ol-h}
\begin{itemize}
\item[(i)] $\Omega$ is a bounded domain in $\R^d$ for some $d \in \N$ with smooth boundary $\partial \Omega$ 
and, moreover, $p \in [2,\infty)$, $q \in (1,2]$ are dual exponents: $1/p + 1/q = 1$
\item[(ii)] $\ol{g} \in C^1(\R,\R)$ and there exist constants $\alpha_1, \alpha_2, \kappa, \lambda \in (0,\infty)$ such that
\begin{align} \label{eq:growth-cond-ol-g}
-\kappa - \alpha_1 |r|^p \le \ol{g}(r)r \le \kappa - \alpha_2 |r|^p
\qquad \text{and} \qquad
\ol{g}'(r) \le \lambda
\qquad (r \in \R)
\end{align}
and, moreover, $\ol{h} \in L^q_{\mathrm{loc}}(\R^+_0,L^q(\Omega))$. 
\end{itemize}
\end{cond}

A bit more explicitly, the first two inequalities in~\eqref{eq:growth-cond-ol-g} mean that $\ol{g}|_{(0,\infty)}$ 
lies between $r \mapsto -\kappa/|r| - \alpha_1 |r|^{p-1}$ and $r \mapsto \kappa/|r| - \alpha_2 |r|^{p-1}$ 
and that $\ol{g}|_{(-\infty,0)}$ lies between $r \mapsto -\kappa/|r| + \alpha_2 |r|^{p-1}$ and $r \mapsto \kappa/|r| + \alpha_1 |r|^{p-1}$. 
A simple class of functions $\ol{g}$ satisfying the three inequalities from~\eqref{eq:growth-cond-ol-g} 
is given by the polynomials of odd degree with negative leading coefficient:
\begin{align*}
\ol{g}(r) = \sum_{i=0}^{2m-1} c_i r^{i} \qquad (r \in \R)
\end{align*} 
with $c_{2m-1} < 0$, where $m \in \N$. (Choose $p := 2m$.) 
In particular, the nonlinearity of the Chaffee--Infante equation given by $\ol{g}(r) := -r^3 + r$ falls into that class (Section~11.5 of~\cite{Robinson}). 
\smallskip 

Suppose that Condition~\ref{cond:ol-g,ol-h} is satisfied and let $s \in \R^+_0$ and $y_s \in X$. A function $y \in C([s,\infty),X)$ is 
called a \emph{global weak solution of 
\eqref{eq:ibvp-general}} iff $y(s) = y_s$ and for every $T \in (s,\infty)$ one has
\begin{align}
y|_{[s,T]} \in L^2([s,T],H^1_0(\Omega)) \cap L^p([s,T], L^p(\Omega))
\end{align}
and there exists a (then unique) $z \in L^2([s,T],H^1_0(\Omega)^*) + L^q([s,T], L^q(\Omega))$ such that
\begin{align} \label{eq:def-weak-sol}
\int_s^T \big( z(t), \phi(t)\big) \d t &= -\int_s^T \int_{\Omega} \nabla y(t)(\zeta) \cdot \nabla \phi(t)(\zeta) \d \zeta \d t + \int_s^T \int_{\Omega} \ol{g}\big( y(t)(\zeta) \big) \, \phi(t)(\zeta) \d \zeta \d t  \notag \\
&\quad + \int_s^T \int_{\Omega} \ol{h}(t)(\zeta) \, \phi(t)(\zeta)  \d \zeta \d t 
\end{align}
for every $\phi \in L^2([s,T],H^1_0(\Omega)) \cap L^p([s,T], L^p(\Omega))$. See~\cite{VaKa06} or~\cite{KaVa09} and, for more background information, \cite{ChVi96} or~\cite{ChVi}. In this equation, $(\cdot,\cdot \cdot)$ stands for the dual pairing of $H^1_0(\Omega)^* + L^q(\Omega)$ and $H^1_0(\Omega) \cap L^p(\Omega)$, that is, 
\begin{align} \label{eq:dual-pairing}
(z,\phi) = (z_1,\phi)_{H^1_0(\Omega)^*,H^1_0(\Omega)} + (z_2,\phi)_{L^q(\Omega),L^p(\Omega)}
\end{align}
for every $z = z_1+z_2 \in H^1_0(\Omega)^* + L^q(\Omega)$ and $\phi \in H^1_0(\Omega) \cap L^p(\Omega)$, where $(\cdot,\cdot\cdot)_{H^1_0(\Omega)^*,H^1_0(\Omega)}$ and $(\cdot,\cdot\cdot)_{L^q(\Omega),L^p(\Omega)}$ denote the respective dual pairings. See~\cite{BeLoe} (Theorem~2.7.1) and \cite{DiUh} (Theorem~IV.1.1 and Corollary~III.2.13), for instance, to get that $H^1_0(\Omega)^* + L^q(\Omega)$, $H^1_0(\Omega) \cap L^p(\Omega)$ and
\begin{align*}
L^2([s,T],H^1_0(\Omega)^*) + L^q([s,T], L^q(\Omega)), \quad L^2([s,T],H^1_0(\Omega)) \cap L^p([s,T], L^p(\Omega))
\end{align*} 
are dual to each other. 
We point out that if $y$ is a global weak solution to~\eqref{eq:ibvp-general}, then for every $T \in (s,\infty)$ there is only one $z \in L^2([s,T],H^1_0(\Omega)^*) + L^q([s,T], L^q(\Omega))$ satisfying~\eqref{eq:def-weak-sol}. And this $z$ is called the \emph{weak or generalized derivative} of $y|_{[s,T]}$. It is denoted by $\partial_t y|_{[s,T]}$ or simply by $\partial_t y$ in the following.

\begin{lm} \label{lm:weak-sol-general}
Suppose that Condition~\ref{cond:ol-g,ol-h} is satisfied and let $s \in \R^+_0$ and $y_s \in X$. Then the initial boundary value problem~\eqref{eq:ibvp-general} has a unique global weak solution $y$ and, moreover, $t \mapsto \norm{y(t)}^2$ is absolutely continuous (hence differentiable almost everywhere) with
\begin{align}
\ddt \norm{y(t)}^2 = 2 \big( \partial_t y(t),y(t) \big) 
\end{align} 
for almost every $t \in [s,\infty)$, where $(\cdot,\cdot \cdot)$ is the dual pairing from~\eqref{eq:dual-pairing}. 
\end{lm}

\begin{proof}
It is clear from the first two inequalities in~\eqref{eq:growth-cond-ol-g} that
\begin{align} \label{eq:estimate-g(r)r}
|\ol{g}(r)r| \le \kappa + \alpha_1|r|^p \qquad (r \in \R).
\end{align} 
Since $\sup_{|r|\le 1} |\ol{g}(r)| < \infty$ by the continuity of $\ol{g}$, it follows from~\eqref{eq:estimate-g(r)r} that for some constant $C_1 \in (0,\infty)$
\begin{align} \label{eq:g(y)-in-Lq}
|\ol{g}(r)| \le C_1 (1+|r|^{p-1}) \qquad (r \in \R)
\end{align}
and therefore condition~(2) from~\cite{VaKa06} is satisfied. Also, in view of the second and third inequalities in~\eqref{eq:growth-cond-ol-g}, condition~(3) and condition~(4) from~\cite{VaKa06} with $M=0$ are satisfied. 
Consequently, the assertions of the lemma follow from the remarks made in Section~2 (up to Remark~1) of~\cite{VaKa06}. 
%
\end{proof}

With this lemma at hand, it is easy to see that the initial boundary value problem~\eqref{eq:ibvp} generates a semiprocess family $(S_u)_{u\in\mathcal{U}}$ on $X$ (Lemma~\ref{lm:semiproc}).  A \emph{semiprocess family on $X$} is a family of maps $S_u: \Delta \times X \to X$ for every $u \in \mathcal{U}$ such that
\begin{gather}
S_u(s,s,x) = x \qquad \text{and} \qquad S_u\big(t,s, S_u(s,r,x) \big) = S(t,r,x) \label{eq:def-semiproc-1}\\
S_u(t+\tau,s+\tau,x) = S_{u(\cdot+\tau)}(t,s,x) \label{eq:def-semiproc-2}
\end{gather} 
for all $(t,s), (s,r) \in \Delta$, $\tau \in \R^+_0$, $x \in X$ and $u \in \mathcal{U}$, where we used the abbreviation $\Delta := \{(s,t) \in \R^+_0 \times \R^+_0: t \ge s\}$. See~\cite{ChVi}, for instance, for more information on semiprocess families. 

\begin{cond} \label{cond:g,h}
\begin{itemize}
\item[(i)] $\Omega$ is a bounded domain in $\R^d$ for some $d \in \N$ with smooth boundary $\partial \Omega$ 
and, moreover, $p \in [2,\infty)$
\item[(ii)] $g \in C^1(\R,\R)$ and there exist constants $\alpha_1, \alpha_2, \kappa, \lambda \in (0,\infty)$ such that
\begin{align} \label{eq:cond-g}
-\kappa - \alpha_1 |r|^p \le g(r)r \le \kappa - \alpha_2 |r|^p
\qquad \text{and} \qquad
g'(r) \le \lambda
\qquad (r \in \R)
\end{align}
and, moreover, $h \in X \setminus \{0\}$. 
\end{itemize}
\end{cond}

\begin{lm} \label{lm:semiproc}
Suppose that Condition~\ref{cond:g,h} is satisfied. Then for every $s \in \R^+_0$ and every $(y_s,u) \in X \times \mathcal{U}$ the initial boundary value problem~\eqref{eq:ibvp} has a unique global weak solution $y(\cdot,s,y_s,u)$. Additionally, $(S_u)_{u\in\mathcal{U}}$ defined by
\begin{align} \label{eq:S_u-def}
S_u(t,s,y_s) := y(t,s,y_s,u)
\end{align}
is a semiprocess family on $X$. 
\end{lm}

\begin{proof}
In order to see the unique global weak solvability, simply apply Lemma~\ref{lm:weak-sol-general} with $\ol{g}:=g$ and with $\ol{h} \in L^2_{\mathrm{loc}}(\R^+_0,L^2(\Omega)) \subset L^q_{\mathrm{loc}}(\R^+_0,L^q(\Omega))$ defined by $\ol{h}(t)(\zeta) := h(\zeta) u(t)$. 
In order to see the semiprocess property, use the definition of weak solutions and the uniqueness statement from Lemma~\ref{lm:weak-sol-general}. 
\end{proof}

In the following, $(S_u)_{u\in\mathcal{U}}$ will always denote the semiprocess family from the previous lemma. Also, we will often refer to $(S_u)_{u\in \mathcal{U}}$ and $S_0$ as the disturbed and the undisturbed system, respectively. 
In proving our local input-to-state stability result, the following estimates will play an important role. 

\begin{lm} \label{lm:semiproc-estimates}
Suppose that Condition~\ref{cond:g,h} is satisfied. Then 
\begin{align}
\norm{S_0(t,0,y_{01})-S_0(t,0,y_{02})} &\le \e^{\lambda t} \norm{y_{01}-y_{02}} \qquad (t \in \R^+_0) \label{eq:semiproc-estimate-1}\\
\norm{S_u(t,0,y_0)-S_0(t,0,y_0)} &\le 2 \e^{2 \lambda} \norm{h} \norm{u}_{\infty} t \qquad (t \in [0,1]) \label{eq:semiproc-estimate-2}
\end{align} 
for all $y_0, y_{01}, y_{02} \in X$ and all $u \in \mathcal{U}$.
\end{lm}

\begin{proof}
As a first step, we show that for every $y_{01},y_{02} \in X$ and $u \in \mathcal{U}$ the function 
\begin{align} \label{eq:y-12-u}
y_{12}^{u} := y_1^{u}-y_2^0 \qquad \text{with} \qquad y_1^{u} := S_u(\cdot,0,y_{01}) \qquad \text{and} \qquad y_2^0 := S_0(\cdot,0,y_{02})
\end{align}
is a global weak solution of the initial boundary value problem
\begin{align} \label{eq:semiproc-estimates-ibvp-general}
\begin{split}
\partial_t y(t,\zeta) &= \Delta y(t,\zeta) + \ol{g}(y(t,\zeta)) + \ol{h}(t,\zeta)  \qquad ((t,\zeta) \in [0,\infty) \times \Omega) \\
y(t,\cdot)|_{\partial \Omega} &= 0 \qquad \text{and} \qquad y(0,\cdot) = y_{01}-y_{02} \qquad (t \in [0,\infty)),
\end{split} 
\end{align}
where $\ol{g} := g$ and $\ol{h}(t)(\zeta) := g(y_1^{u}(t)(\zeta)) - g(y_2^{0}(t)(\zeta)) - g(y_{12}^{u}(t)(\zeta)) + h(\zeta)u(t)$. 
So, let $y_{01},y_{02} \in X$ and $u \in \mathcal{U}$ and adopt the abbreviations from~\eqref{eq:y-12-u}. It is not difficult -- using~\eqref{eq:g(y)-in-Lq} and $q(p-1)=p$ -- to see from Condition~\ref{cond:g,h} that with $\ol{g}$, $\ol{h}$ as defined above, Condition~\ref{cond:ol-g,ol-h} is satisfied.  
Since $y_1^{u}$, $y_2^0$ are global weak solutions, we have $y_{12}^{u} \in C(\R^+_0,X)$ and for every $T \in (0,\infty)$ we have
\begin{align*}
y_{12}^{u}|_{[0,T]} \in L^2([0,T],H^1_0(\Omega)) \cap L^p([0,T], L^p(\Omega))
\end{align*}
and $\partial_t y_{1}^{u}|_{[0,T]} - \partial_t y_{2}^{0}|_{[0,T]} \in L^2([0,T],H^1_0(\Omega)^*) + L^q([0,T], L^q(\Omega))$ as well as
\begin{align} \label{eq:semiproc-estimates-step-1}
\int_0^T \big( \partial_t y_{1}^{u}(t) &- \partial_t y_{2}^{0}(t), \phi(t) \big) \d t 
= - \int_0^T \int_{\Omega} \nabla y_{12}^{u}(t)(\zeta) \cdot \nabla \phi(t)(\zeta) \d\zeta \d t \notag \\
&+  \int_0^T \int_{\Omega} \ol{g}\big( y_{12}^{u}(t)(\zeta) \big) \, \phi(t)(\zeta) \d \zeta \d t  + \int_0^T \int_{\Omega} \ol{h}(t)(\zeta) \, \phi(t)(\zeta)  \d \zeta \d t
\end{align}
for every $\phi \in L^2([0,T],H^1_0(\Omega)) \cap L^p([0,T], L^p(\Omega))$. And therefore, $y_{12}^{u}$ is a weak solution of~\eqref{eq:semiproc-estimates-ibvp-general}, as desired.
\smallskip

As a second step, we show that for every $y_{01},y_{02} \in X$ and $u \in \mathcal{U}$ the function $y_{12}^{u}$ from~\eqref{eq:y-12-u} satisfies the estimate
\begin{align} \label{eq:semiproc-estimates-step-2}
\sup_{T \in [0,t]} \norm{y_{12}^{u}(T)}^2 \le \e^{2\lambda t} \Big( \norm{y_{01}-y_{02}}^2 + 2\norm{h}\norm{u}_{\infty} \cdot t \cdot \sup_{T \in [0,t]}  \norm{y_{12}^{u}(T)} \Big) 
\end{align}
for every $t \in \R^+_0$. 
%
Indeed, by the first step and Lemma~\ref{lm:weak-sol-general}, the function $t \mapsto \norm{y_{12}^{u}(t)}^2$ is absolutely continuous with 
\begin{align*}
\ddt \frac{\norm{y_{12}^{u}(t)}^2}{2} = \big( \partial_t y_{12}^{u}(t), y_{12}^{u}(t) \big) = \big( \partial_t y_{1}^{u}(t) - \partial_t y_{2}^{0}(t), y_{12}^{u}(t) \big) 
\end{align*}
for almost every $t \in \R^+_0$. And therefore, by virtue of~\eqref{eq:semiproc-estimates-step-1} with $\phi := y_{12}^{u}$, we get
\begin{align} \label{eq:semiproc-estimates-step-2,1}
&\frac{\norm{y_{12}^{u}(T)}^2}{2} - \frac{\norm{y_{12}^{u}(0)}^2}{2}
=
\int_0^T \big( \partial_t y_{1}^{u}(t) - \partial_t y_{2}^{0}(t) , y_{12}^{u}(t) \big)  \d t \notag \\
&\qquad \le 
\int_0^T \int_{\Omega} \big( g(y_1^{u}(t)(\zeta)) - g(y_2^{0}(t)(\zeta)) \big) y_{12}^{u}(t)(\zeta) \d\zeta \d t + \int_0^T \int_{\Omega} h(\zeta) u(t) y_{12}^{u}(t)(\zeta) \d\zeta \d t \notag \\
&\qquad \le 
\lambda \int_0^T \norm{y_{12}^{u}(t)}^2 \d t + \norm{h}\norm{u}_{\infty} \int_0^T \norm{y_{12}^{u}(t)} \d t
\end{align}
for every $T \in (0,\infty)$. In the last inequality, we used that $(g(r)-g(s)) (r-s) \le \lambda |r-s|^2$ for all $r,s \in \R$ due to~\eqref{eq:cond-g}. 
So, for every $t_0 \in (0,\infty)$, we obtain 
\begin{align*}
\norm{y_{12}^{u}(T)}^2 \le \norm{y_{01}-y_{02}}^2 + 2\norm{h}\norm{u}_{\infty} \cdot t_0 \cdot \sup_{t\in[0,t_0]} \norm{y_{12}^{u}(t)} + 2 \lambda \int_0^T \norm{y_{12}^{u}(t)}^2 \d t
\end{align*}
for every $T \in [0,t_0]$. And from this, in turn, 
the claimed estimate~\eqref{eq:semiproc-estimates-step-2} immediately follows by Gr\"onwall's lemma.  
\smallskip

As a third step, it is now easy to conclude the desired estimates~\eqref{eq:semiproc-estimate-1} and~\eqref{eq:semiproc-estimate-2} from the second step.
Indeed, \eqref{eq:semiproc-estimate-1} immediately follows from~\eqref{eq:semiproc-estimates-step-2} with the special choice $u := 0 \in \mathcal{U}$ and~\eqref{eq:semiproc-estimate-2} follows from~\eqref{eq:semiproc-estimates-step-2} with the special choice $y_{01} = y_{02} := y_0 \in X$.  
\end{proof}

We remark for later reference that our semiprocess family $(S_u)_{u\in\mathcal{U}}$, like any other semiprocess family~\cite{ScKaDa19-wAG},
satisfies the following so-called cocycle property:
\begin{align}
S_u(t+\tau,0,x) = S_{u(\cdot+\tau)}\big(t,0,S_u(\tau,0,x)\big)
\end{align}
for all $t, \tau \in \R^+_0$, $x \in X$ and $u \in \mathcal{U}$. (Just combine~\eqref{eq:def-semiproc-1} and~\eqref{eq:def-semiproc-2} to see this.) In particular, $S_0$ satisfies the following (nonlinear) semigroup property~\cite{Miyadera}: 
\begin{align} \label{eq:S0-sgr}
S_0(t+\tau,0,x) = S_0\big(t,0,S_0(\tau,0,x)\big) \qquad (t,\tau \in \R^+_0 \text{ and } x \in X).
\end{align} 

We conclude this section with some remarks on the asymptotic behavior 
of this semigroup $S_0$ in terms of attractors~\cite{Robinson}, \cite{Temam}. 
A \emph{global attractor} of $S_0$ is a compact subset $\Theta$ of $X$ such that
\begin{itemize}
\item[(i)] $\Theta$ is invariant under $S_0$, that is, $S_0(t,0,\Theta) = \Theta$ for every $t \in \R^+_0$
\item[(ii)] $\Theta$ is uniformly attractive for $S_0$, that is, for every bounded subset $B \subset X$ one has
\begin{align} \label{eq:uniform-attractivity-def}
\dist\big( S_0(t,0,B), \Theta \big) = \sup_{x\in B} \norm{S_0(t,0,x)}_{\Theta} \longrightarrow 0 \qquad (t \to \infty).
\end{align}
\end{itemize} 
It directly follows from this definition that a global attractor of $S_0$ is minimal among all closed uniformly attractive sets of $S_0$ and maximal among all bounded invariant sets of $S_0$. And from this, in turn, it immediately follows that if $S_0$ has any global attractor then it is already unique.

\begin{lm} \label{lm:S0-UGAS}
Suppose that Condition~\ref{cond:g,h} is satisfied. Then the undisturbed system $S_0$ has a unique global attractor $\Theta$ and, moreover, $\Theta$ is uniformly globally asymptotically stable for $S_0$, that is, there exists a comparison function $\beta_0 \in \mathcal{KL}$ such that
\begin{align} \label{eq:S0-UGAS}
\norm{S_0(t,0,x)}_{\Theta} \le \beta_0(\norm{x}_{\Theta},t) \qquad (t\in \R^+_0 \text{ and } x \in X). 
\end{align}
\end{lm}

\begin{proof}
It is well-known that $S_0$ has a global attractor $\Theta$ (by Theorem~11.4 of~\cite{Robinson}, for instance) and that global attractors when existent are already unique (by the remarks preceding the lemma). 
So, we have only 
to show that $\Theta$ is uniformly globally asymptotically stable for $S_0$. And in order to do so, we will proceed in three steps, applying results from~\cite{Mi17} to the system $S_0 = (S_u)_{u\in\mathcal{U}_0}$ with trivial disturbance space $\mathcal{U}_0 := \{0\}$. (In this context, it should be noticed that by~\eqref{eq:S0-sgr} and the continuity of weak solutions, $(S_u)_{u\in\mathcal{U}_0}$ is a forward-complete system in the sense of~\cite{Mi17}, \cite{MiWi18}, \cite{Sc19-wISS}.)
\smallskip

As a first step, we show that $\Theta$ is uniformly globally stable for $(S_u)_{u\in\mathcal{U}_0} = S_0$, 
that is, there exists a comparison function $\sigma_0 \in \mathcal{K}$ such that
\begin{align}
\norm{S_0(t,0,x)}_{\Theta} \le \sigma_0(\norm{x}_{\Theta}) \qquad (t \in \R^+_0)
\end{align}
for every $x \in X$ (Definition~2.8 of~\cite{Mi17}). 
Indeed, it immediately follows from the invariance of $\Theta$ under $S_0$ and from the estimate~\eqref{eq:semiproc-estimate-1} that for every $\eps > 0$ and every $T \in (0,\infty)$ there exists a $\delta \in (0,1]$ such that
\begin{align*}
\norm{S_0(t,0,x)}_{\Theta} \le \inf_{\theta \in\Theta} \norm{S_0(t,0,x)-S_0(t,0,\theta)} < \eps \qquad (t \in [0,T] \text{ and } x \in B_{\delta}(\Theta)).  
\end{align*}
And from this and the uniform attractivity~\eqref{eq:uniform-attractivity-def} of $\Theta$ for $S_0$ (with $B := B_1(\Theta)$), in turn, it follows that for every $\eps > 0$ there exists a $\delta > 0$ such that
\begin{align} \label{eq:Theta-ULS}
\norm{S_0(t,0,x)}_{\Theta} < \eps \qquad (t \in \R^+_0)
\end{align}
for every $x \in B_{\delta}(\Theta)$. 
%
%
Also, it is well-known that 
\begin{align} \label{eq:dissip-estimate-S0}
\norm{S_0(t,0,x)}^2 \le \e^{-2\omega t} \norm{x}^2 + \frac{\lambda |\Omega|}{\omega} \qquad (t \in \R^+_0)
\end{align}
for all $x \in X$, where $\omega \in (0,\infty)$ is the smallest eigenvalue of $-\Delta$, the negative Dirichlet Laplacian on $\Omega$. (See the very last equation on page 286 of~\cite{Robinson}, for instance.) Since $\norm{S_0(t,0,x)}_{\Theta} \le \norm{S_0(t,0,x)} + \norm{\Theta}$ and $\norm{x} \le \norm{x}_{\Theta} + \norm{\Theta}$ with $\norm{\Theta} := \sup_{\theta \in \Theta} \norm{\theta}$, it follows from~\eqref{eq:dissip-estimate-S0}  that there exists a comparison function $\sigma \in \mathcal{K}$ and a constant $c \in (0,\infty)$ such that
\begin{align} \label{eq:Theta-Lagrange-stable}
\norm{S_0(t,0,x)}_{\Theta} \le \sigma(\norm{x}_{\Theta}) + c \qquad (t \in \R^+_0)
\end{align} 
for every $x \in X$. 
In the terminology of~\cite{Mi17}, the relations~\eqref{eq:Theta-ULS} and~\eqref{eq:Theta-Lagrange-stable} mean that $\Theta$ is uniformly locally stable and Lagrange-stable for $(S_u)_{u\in\mathcal{U}_0}$, respectively. And therefore, $\Theta$ is uniformly globally stable for $(S_u)_{u\in\mathcal{U}_0} = S_0$ by virtue of Remark~2.9 of~\cite{Mi17}, as desired. 
\smallskip

As a second step, we show 
that $\Theta$ is uniformly globally attractive for $(S_u)_{u\in\mathcal{U}_0} = S_0$, that is, for every $\eps >0$ and $r>0$ there exists a time $\tau(\eps,r) \in \R^+_0$ such that 
\begin{align} \label{eq:Theta-UGATT}
\norm{S_0(t,0,x)}_{\Theta} < \eps \qquad (t \ge \tau(\eps,r))
\end{align} 
for every $x \in \ol{B}_r(\Theta)$ (Definition~2.8 of~\cite{Mi17}). 
Indeed, this immediately follows from the uniform attractivity~\eqref{eq:uniform-attractivity-def} of $\Theta$ for $S_0$ with $B := \ol{B}_r(\Theta)$. 
\smallskip

As a third step, we can now conclude the desired uniform global asymptotic stability of $\Theta$ for $(S_u)_{u\in\mathcal{U}_0} = S_0$
from Theorem~4.2 of~\cite{Mi17} and the first two steps. 
\end{proof}

\section{A local input-to-state stability result}

In this section, we establish our local input-to-state stability result for the disturbed reaction-diffusion system~\eqref{eq:react-diffus-eq}. We begin by showing that the undisturbed system~\eqref{eq:react-diffus-eq, undisturbed} has a local Lyapunov function and, for that purpose, we will 
argue in a similar way as~\cite{Henry} (Theorem~4.2.1).

\begin{lm} \label{lm:L-fct}
Suppose that Condition~\ref{cond:g,h} is satisfied and let $\Theta$ be the global attractor of the undisturbed system $S_0$. 
Then for every $r_0 > 0$ there exists a Lipschitz continuous function $V: \ol{B}_{r_0}(\Theta) \to \R^+_0$ with Lipschitz constant $1$ and comparison functions $\ul{\psi}, \ol{\psi}, \alpha \in \mathcal{K}_{\infty}$ such that 
\begin{gather}
\ul{\psi}(\norm{x}_{\Theta}) \le V(x) \le \ol{\psi}(\norm{x}_{\Theta}) \qquad (x \in \ol{B}_{r_0}(\Theta)) \label{eq:V-coercive-Lfct}\\
\dot{V}_0(x) := \varlimsup_{t\to 0+} \frac{1}{t}\big( V(S_0(t,0,x))-V(x) \big)  \le -\alpha(\norm{x}_{\Theta}) \qquad (x \in B_{r_0}(\Theta)). \label{eq:V-Dini-derivative}
\end{gather}
\end{lm}

\begin{proof}
Choose an arbitrary $r_0 \in (0,\infty)$ and fix it for the rest of the proof. Also, choose $\beta_0 \in \mathcal{KL}$ as in Lemma~\ref{lm:S0-UGAS} and, for every $\eps > 0$, 
let $T(\eps) = T_{r_0}(\eps)$ be a time such that
\begin{align} \label{eq:Lfct-def-T(eps)}
\beta_0(r_0,t) \le \eps \qquad (t \in [T(\eps),\infty)).
\end{align} 
Set now, for every given $\eps > 0$,
\begin{align} \label{eq:Lfct-def-V-eps}
V^{\eps}(x) := \e^{-(\lambda+c_0)T(\eps)} \sup_{t \in [0,\infty)} \Big( \e^{c_0 t} \, \eta_{\eps}\big( \norm{S_0(t,0,x)}_{\Theta} \big) \Big)
\qquad (x \in \ol{B}_{r_0}(\Theta)),
\end{align}
where $c_0 \in (0,\infty)$ is an arbitrary constant (which is fixed throughout the proof) and $\eta_{\eps}(r) := \max\{ 0,r-\eps\}$ for every $r \in \R^+_0$. 
In view of~\eqref{eq:S0-UGAS} and~\eqref{eq:Lfct-def-T(eps)}, the supremum in~\eqref{eq:Lfct-def-V-eps} for $x \in \ol{B}_{r_0}(\Theta)$ actually extends only over a compact interval, namely
\begin{align} \label{eq:Lfct-V-eps-supremum-only-over-compact-interval}
V^{\eps}(x) = \e^{-(\lambda+c_0)T(\eps)} \sup_{t \in [0,T(\eps)]} \Big( \e^{c_0 t} \, \eta_{\eps}\big( \norm{S_0(t,0,x)}_{\Theta} \big) \Big)
\qquad (x \in \ol{B}_{r_0}(\Theta)).
\end{align}
In particular, $V^{\eps}: \ol{B}_{r_0}(\Theta) \to \R^+_0$ is a well-defined map (with finite values) and
\begin{align} \label{eq:Lfct-V-eps-le-beta0}
V^{\eps}(x) \le \e^{-\lambda T(\eps)}  \sup_{t \in [0,T(\eps)]} \Big(  \eta_{\eps}\big( \norm{S_0(t,0,x)}_{\Theta} \big) \Big)
\le \beta_0(\norm{x}_{\Theta},0) 
\qquad (x \in \ol{B}_{r_0}(\Theta))
\end{align}
because $\eta_{\eps}(r) \le r$ for all $r \in \R^+_0$. 
Since, moreover, 
$|\eta_{\eps}(r)-\eta_{\eps}(s)| \le |r-s|$ for all $r,s \in \R^+_0$, 
we see from~\eqref{eq:Lfct-V-eps-supremum-only-over-compact-interval} and~\eqref{eq:semiproc-estimate-1} that 
\begin{align} \label{eq:Lfct-V-eps-Lipschitz-with-const-1}
|V^{\eps}(x) &- V^{\eps}(y)| 
\le 
\e^{-(\lambda+c_0)T(\eps)} \sup_{t \in [0,T(\eps)]} \Big| \e^{c_0 t} \, \eta_{\eps}\big( \norm{S_0(t,0,x)}_{\Theta}\big) - \e^{c_0 t} \, \eta_{\eps}\big( \norm{S_0(t,0,y)}_{\Theta} \big) \Big| \notag \\
&\le 
\e^{-\lambda T(\eps)} \sup_{t \in [0,T(\eps)]} \Big|  \norm{S_0(t,0,x)}_{\Theta} - \norm{S_0(t,0,y)}_{\Theta}  \Big| \\
&\le 
\e^{-\lambda T(\eps)} \sup_{t \in [0,T(\eps)]} \norm{ S_0(t,0,x) - S_0(t,0,y) }
\le \norm{x-y}  
\qquad (x,y \in \ol{B}_{r_0}(\Theta)).  \notag
\end{align}
(In the first inequality above, we used the elementary fact that $|\sup_{t \in I} a_t - \sup_{t\in I} b_t| \le \sup_{t\in I} |a_t-b_t|$ for arbitrary bounded functions $t \mapsto a_t, b_t$ on an arbitrary set $I$, and in the third inequality above, we used the elementary fact that $|\norm{\xi}_{\Theta}-\norm{\eta}_{\Theta}| \le \norm{\xi-\eta}$ for arbitrary $\xi,\eta \in X$.) 
Additionally, for every $x \in B_{r_0}(\Theta)$, we have $S_0(\tau,0,x) \in B_{r_0}(\Theta)$ for $\tau$ small enough and thus, by~\eqref{eq:Lfct-def-V-eps} and the semigroup property~\eqref{eq:S0-sgr},
\begin{align*}
V^{\eps}(S_0(\tau,0,x)) = \e^{-(\lambda+c_0)T(\eps)} \sup_{t \in [0,\infty)} \Big( \e^{c_0 t} \, \eta_{\eps}\big( \norm{S_0(t+\tau,0,x)}_{\Theta} \big) \Big) 
\le \e^{-c_0 \tau} V^{\eps}(x)
\end{align*}
for every $x \in B_{r_0}(\Theta)$ and all sufficiently small times $\tau$. Consequently,
\begin{align} \label{eq:Lfct-Dini-derivative-V-eps}
\dot{V}_0^{\eps}(x) = \varlimsup_{\tau \to 0+} \frac{1}{\tau} \big( V^{\eps}(S_0(\tau,0,x)) - V^{\eps}(x) \big) \le -c_0 V^{\eps}(x)
\qquad (x \in B_{r_0}(\Theta)). 
\end{align}
With the help of the auxiliary functions $V^{\eps}$, we can now construct a function $V: \ol{B}_{r_0}(\Theta) \to \R^+_0$ with the desired properties. Indeed, let 
\begin{align} \label{eq:Lfct-def-V}
V(x) := \sum_{k=1}^{\infty} 2^{-k} V^{1/k}(x) \qquad (x \in \ol{B}_{r_0}(\Theta)).
\end{align}
We then conclude from~\eqref{eq:Lfct-V-eps-le-beta0}, \eqref{eq:Lfct-V-eps-Lipschitz-with-const-1}, \eqref{eq:Lfct-Dini-derivative-V-eps} that
\begin{gather}
V(x) \le \beta_0(\norm{x}_{\Theta},0) \qquad (x \in \ol{B}_{r_0}(\Theta)), \label{eq:Lfct-bounded-above-by-beta0}\\
|V(x)-V(y)| \le \sum_{k=1}^{\infty} 2^{-k} |V^{1/k}(x)-V^{1/k}(y)| \le \norm{x-y}
\qquad (x,y \in \ol{B}_{r_0}(\Theta)),  \label{eq:Lfct-Lipschitz-with-constant-1}\\
\dot{V}_0(x) \le \sum_{k=1}^{\infty} 2^{-k} \dot{V}_0^{1/k}(x) \le -c_0 V(x)  \qquad (x \in B_{r_0}(\Theta)). \label{eq:Lfct-Dini-derivative}
\end{gather}
Since $\sup_{t\in[0,\infty)} (\e^{c_0 t} \eta_{1/k}(\norm{S_0(t,0,x)}_{\Theta})) \ge \eta_{1/k}(\norm{x}_{\Theta})$ for all $x \in X$, we also conclude from~\eqref{eq:Lfct-def-V-eps} and~\eqref{eq:Lfct-def-V} that
\begin{align} \label{eq:Lfct-bounded-below}
V(x) \ge \sum_{k=1}^{\infty} 2^{-k} \e^{-(\lambda + c_0) T(1/k)} \eta_{1/k}(\norm{x}_{\Theta}) \qquad (x \in \ol{B}_{r_0}(\Theta).
\end{align}
In view of these estimates, we now define the comparison functions $\ol{\psi}$, $\ul{\psi}$ and $\alpha$  in the following way:
\begin{align*}
\ol{\psi}(r) := \beta_0(r,0) + r \qquad \text{and} \qquad \ul{\psi}(r) := \sum_{k=1}^{\infty} 2^{-k} \e^{-(\lambda + c_0) T(1/k)} \eta_{1/k}(r)
\end{align*}
and $\alpha(r) := c_0 \ul{\psi}(r)$ for $r \in \R^+_0$. It is easy to verify that $\ol{\psi}$, $\ul{\psi}$ and hence $\alpha$ belong to $\mathcal{K}_{\infty}$. And, moreover, by virtue of~\eqref{eq:Lfct-bounded-above-by-beta0}, \eqref{eq:Lfct-Lipschitz-with-constant-1}, \eqref{eq:Lfct-Dini-derivative}, \eqref{eq:Lfct-bounded-below}, the desired estimates~\eqref{eq:V-coercive-Lfct} and~\eqref{eq:V-Dini-derivative} follow. 
\end{proof}

It should be noticed that the functions $V, \ul{\psi}, \alpha$ constructed in the proof above all depend on the chosen radius $r_0 \in (0,\infty)$ because these functions are defined in terms of the times $T(\eps) = T_{r_0}(\eps)$ from~\eqref{eq:Lfct-def-T(eps)}. 
With the next lemma, we show that the local Lyapunov function $V$ for the undisturbed system is also a local input-to-state Lyapunov function 
for the disturbed system w.r.t.~$\Theta$. (See~\cite{DaMi13} for the definition of local input-to-state Lyapunov functions w.r.t.~an equilibrium point.)

\begin{lm} \label{lm:lISS-L-fct}
Suppose that Condition~\ref{cond:g,h} is satisfied and let $\Theta$ be the global attractor of the undisturbed system $S_0$. Also, let $r_0 > 0$ and let $V: \ol{B}_{r_0}(\Theta) \to \R_0^+$ be chosen as in the previous lemma. Then there exist comparison functions $\alpha, \sigma \in \mathcal{K}$ such that 
for every $u \in \mathcal{U}$
\begin{align*}
\dot{V}_u(x) := \varlimsup_{t\to0+} \frac{1}{t}\big( V(S_u(t,0,x))-V(x) \big) \le -\alpha(\norm{x}_{\Theta}) + \sigma(\norm{u}_{\infty})
\qquad (x \in B_{r_0}(\Theta)).
\end{align*}
\end{lm}

\begin{proof}
Choose $\alpha = \alpha_{r_0} \in \mathcal{K}_{\infty}$ as in Lemma~\ref{lm:L-fct} and define $\sigma \in \mathcal{K}_{\infty}$ 
by $\sigma(r) := 2\e^{2\lambda} \norm{h} r$ for all $r \in \R^+_0$. 
We then see from Lemma~\ref{lm:L-fct} and from \eqref{eq:semiproc-estimate-2} 
that for every $x \in B_{r_0}(\Theta)$ and every $u \in \mathcal{U}$
\begin{align}
\dot{V}_u(x) \le \varlimsup_{t\to 0+} \frac{1}{t}\big( V(S_0(t,0,x))-V(x) \big) + \varlimsup_{t\to 0+} \frac{1}{t}\big( V(S_u(t,0,x))-V(S_0(t,0,x)) \big) \notag \\
\le -\alpha(\norm{x}_{\Theta}) + \varlimsup_{t\to 0+} \frac{1}{t} \norm{ S_u(t,0,x) - S_0(t,0,x) } 
\le -\alpha(\norm{x}_{\Theta}) + \sigma(\norm{u}_{\infty}),
\end{align}
as desired. 
\end{proof}

With these lemmas at hand, we can now establish the local input-to-state stability of the disturbed reaction-diffusion system~\eqref{eq:react-diffus-eq} w.r.t.~the global attractor of the undisturbed system~\eqref{eq:react-diffus-eq, undisturbed}. 
It is an open question -- left to future research -- whether this result can actually be extended to a semi-global 
input-to-state stability 
result. See the remarks after the proof for a discussion of the obstacles to such an extension. 

\begin{thm} \label{thm:lISS}
Suppose that Condition~\ref{cond:g,h} is satisfied and let $\Theta$ be the global attractor of the undisturbed system $S_0$. Then the disturbed system $(S_u)_{u\in\mathcal{U}}$ is locally input-to-state stable w.r.t.~$\Theta$, that is, there exist comparison functions $\beta \in \mathcal{KL}$ and $\gamma \in \mathcal{K}$ and radii $r_{0x}, r_{0u} > 0$ such that
\begin{align}
\norm{S_u(t,0,x_0)}_{\Theta} \le \beta(\norm{x_0}_{\Theta},t) + \gamma(\norm{u}_{\infty}) \qquad (t \in \R^+_0)
\end{align}
for all $(x_0,u) \in X \times \mathcal{U}$ with $\norm{x_0}_{\Theta} \le r_{0x}$ and  $\norm{u}_{\infty} \le r_{0u}$.
\end{thm}

\begin{proof}
Choose an arbitrary $r_0 \in (0,\infty)$ and fix it for the entire proof. Also,  take $V = V_{r_0}$ and $\ul{\psi} = \ul{\psi}_{r_0}$, $\ol{\psi}$ as in Lemma~\ref{lm:L-fct}.
It then immediately follows from Lemma~\ref{lm:lISS-L-fct} that there exist comparison functions $\alpha = \alpha_{r_0} \in \mathcal{K}$ and $\chi = \chi_{r_0} \in \mathcal{K}$ such that for all $(x_0,u) \in X \times \mathcal{U}$ with $r_0 \ge \norm{x_0}_{\Theta} \ge \chi(\norm{u}_{\infty})$ one has
\begin{align} \label{eq:lISS-L-fct,implicative}
\dot{V}_u(x_0) \le -\alpha(\norm{x_0}_{\Theta}). 
\end{align}
(Simply choose $\chi(r) := \alpha_0^{-1}(2\sigma_0(r))$ and $\alpha(r) := \alpha_0(r)/2$, where $\alpha_0, \sigma_0 \in \mathcal{K}_{\infty}$ are as in Lemma~\ref{lm:lISS-L-fct}.)
According to the comparison lemma from~\cite{MiIt16} (Corollary~1), we can then choose a comparison function $\ol{\beta} = \ol{\beta}_{\alpha\circ \ol{\psi}^{-1}}$ in such a way that for every $T \in (0,\infty]$ and every function $v \in C([0,T),\R^+_0)$ with
\begin{align*}
\ol{\partial}_t^+ v(t) \le -(\alpha \circ \ol{\psi}^{-1})(v(t)) \qquad (t \in [0,T))
\end{align*} 
one has $v(t) \le \ol{\beta}(v(0),t)$ for all $t\in[0,T)$. 
We now define
\begin{align} \label{eq:def-beta-gamma}
\beta(r,t) := \ul{\psi}^{-1}\big( \ol{\beta}(\ol{\psi}(r),t) \big) 
\qquad \text{and} \qquad
\gamma(r) := \ul{\psi}^{-1}\big( \ol{\psi}(\chi(r)) \big)
\end{align} 
for $r,t \in \R^+_0$ and choose $r_{0x}, r_{0u} \in (0,\infty)$ so small that
\begin{align} \label{eq:def-r0x-and-r0u}
r_{0x} < r_0 \qquad \text{and} \qquad \beta(r_{0x},0) < r_0 
\qquad \text{and} \qquad
\gamma(r_{0u}) < r_0. 
\end{align}
Also, we will write
\begin{align} \label{eq:M_u-def}
M_u := \big\{ x \in \ol{B}_{r_0}(\Theta): V(x) \le \ol{\psi}(\chi(\norm{u}_{\infty})) \big\}
\end{align}
for $u \in \mathcal{U}$.
Clearly, $\beta \in \mathcal{KL}$, $\gamma \in \mathcal{K}$ and $M_u$ is closed for every $u \in \mathcal{U}$. Additionally, for every $u \in \ol{B}_{r_{0u}}(0)$ we have by~\eqref{eq:def-r0x-and-r0u} that
\begin{align} \label{eq:M_u-contained-in-better-sets}
M_u \subset \big\{ x \in \ol{B}_{r_0}(\Theta): \norm{x}_{\Theta} \le \gamma(\norm{u}_{\infty}) \big\} \subset B_{r_0}(\Theta).
\end{align}
%

After these preliminary considerations, we now prove that
\begin{align} \label{eq:lISS-estimate}
\norm{S_u(t,0,x_0)}_{\Theta} \le \beta(\norm{x_0}_{\Theta},t) + \gamma(\norm{u}_{\infty}) \qquad (t \in \R^+_0)
\end{align}
for all $(x_0,u) \in \ol{B}_{r_{0x}}(\Theta) \times \ol{B}_{r_{0u}}(0)$ and thus obtain the desired local input-to-state stability. So, let $(x_0,u) \in \ol{B}_{r_{0x}}(\Theta) \times \ol{B}_{r_{0u}}(0)$ be fixed for the rest of the proof. 
We will distinguish two cases in the following, namely the case where $x_0 \in M_u$ (part (i) of the proof) and the case where $x_0 \notin M_u$ (part (ii) of the proof). 
\smallskip

(i) Suppose we are in the case $x_0 \in M_u$. In order to establish~\eqref{eq:lISS-estimate} in that case, 
we will show -- in two steps -- that for every $t_0 \in [0,\infty)$ one has
\begin{align} \label{eq:M_u-invariant}
S_u(t,t_0,M_u) \in M_u \qquad (t \in [t_0,\infty)). 
\end{align}
So, let $t_0 \in [0,\infty)$ and $x_{t_0} \in M_u$ 
and 
\begin{align} \label{eq:def-T-case(i)}
T := \sup \big\{ T' \in (t_0,\infty): \norm{x(t)}_{\Theta} < r_0 \text{ for all } t \in [t_0,T') \big\}, 
\end{align}
where we use the abbreviation $x(t) := S_u(t,t_0,x_{t_0})$. 
Since $x(t_0) = x_{t_0} \in M_u$ and thus $\norm{x(t_0)}_{\Theta} < r_0$ by~\eqref{eq:M_u-contained-in-better-sets}, we observe 
that $T \in (t_0,\infty]$ and that
\begin{align} \label{eq:norm-initially-less-than-r0,case(i)}
\norm{x(t)}_{\Theta} < r_0 \qquad (t \in [t_0,T)).
\end{align}

As a first step, we show that $x(t) \in M_u$ at least for all $[t_0,T)$. 
Assuming the contrary, we find a $t \in [t_0,T)$ and an $\eps > 0$ such that $V(x(t)) > \ol{\psi}(\chi(\norm{u}_{\infty})) + \eps$. Since $x(t_0) = x_{t_0} \in M_u$ and thus $V(x(t_0)) \le \ol{\psi}(\chi(\norm{u}_{\infty})) + \eps$, 
we observe 
that
\begin{align}
t_1 := \inf \big\{ t \in [t_0,T): V(x(t)) > \ol{\psi}(\chi(\norm{u}_{\infty})) + \eps \big\}
\end{align}
belongs to the interval $(t_0,T)$ and, moreover, $V(x(t_1)) = \ol{\psi}(\chi(\norm{u}_{\infty})) + \eps$. So, 
\begin{align*}
\ol{\psi}(\norm{x(t_1)}_{\Theta}) \ge V(x(t_1)) > \ol{\psi}(\chi(\norm{u}_{\infty})) \ge \ol{\psi}\big(\chi(\norm{u(\cdot+t_1)}_{\infty})\big)
\end{align*} 
and therefore we get by virtue of~\eqref{eq:lISS-L-fct,implicative} that
\begin{align}
\varlimsup_{t\to 0+} \frac{1}{t}\Big( V(x(t_1+t))-V(x(t_1)) \Big) 
&= \varlimsup_{t\to 0+} \frac{1}{t}\Big( V\big(S_{u(\cdot + t_1)}(t,0,x(t_1))\big)-V(x(t_1)) \Big)  \notag \\
&= \dot{V}_{u(\cdot + t_1)}(x(t_1)) \le -\alpha(\norm{x(t_1)}_{\Theta}) < 0.
\end{align}
Consequently, there exists a $\delta >0$ such that $V(x(t_1+t)) \le V(x(t_1)) = \ol{\psi}(\chi(\norm{u}_{\infty})) + \eps$ for all $t \in [0,\delta)$. Contradiction to the definition of $t_1$! 
\smallskip

As a second step, we show that $T = \infty$. Indeed, assuming $T < \infty$, we would get by the first step and continuity 
that even $x(T) \in M_u$ and thus $\norm{x(T)}_{\Theta} < r_0$ by~\eqref{eq:M_u-contained-in-better-sets}. And from this, in turn, it would follow again by continuity that $\norm{x(t)}_{\Theta} < r_0$ for all $t \in [T,T+\delta)$ with some $\delta > 0$. In conjunction with~\eqref{eq:norm-initially-less-than-r0,case(i)}, this would yield a contradiction to the definition~\eqref{eq:def-T-case(i)} of $T$! 
\smallskip

Combining now the first and the second step, we finally obtain the desired invariance~\eqref{eq:M_u-invariant}, which clearly implies~\eqref{eq:lISS-estimate} in the case $x_0 \in M_u$. 
\smallskip

(ii) Suppose we are in the case $x_0 \notin M_u$. In order to establish~\eqref{eq:lISS-estimate} in that case, 
we will show -- in three steps -- that for some $t_0 \in (0,\infty]$ one has
\begin{align} 
\norm{S_u(t,0,x_0)}_{\Theta} &\le \beta(\norm{x_0}_{\Theta},t) \qquad (t\in [0,t_0)) \label{eq:part-(ii),1} \\
\norm{S_u(t,0,x_0)}_{\Theta} &\le \gamma(\norm{u}_{\infty}) \qquad (t \in (t_0,\infty)). \label{eq:part-(ii),2}
\end{align}
Indeed, let $t_0 := \inf \{t \in \R^+_0: x(t) \in M_u\}$ 
and 
\begin{align} \label{eq:def-T-case(ii)}
T := \sup \big\{ T' \in (0,t_0): \norm{x(t)}_{\Theta} < r_0 \text{ for all } t \in [0,T') \big\},
\end{align}
where we use the abbreviation $x(t) := S_u(t,0,x_0)$. (In view of the standard convention $\inf \emptyset := \infty$, we have $t_0 = \infty$ in case $x(t) \notin M_u$ for all $t \in \R^+_0$.) Since $x(0) = x_0 \in (X\setminus M_u) \cap \ol{B}_{r_{0x}}(\Theta)$ and thus $\norm{x(0)}_{\Theta} < r_0$ by~\eqref{eq:def-r0x-and-r0u}, we observe that $t_0 \in (0,\infty]$ and $T \in (0,t_0]$ and that
\begin{align} \label{eq:norm-initially-less-than-r0,case(ii)}
x(t) \notin M_u \qquad (t \in [0,t_0)) \qquad \text{and} \qquad \norm{x(t)}_{\Theta} < r_0 \qquad (t \in [0,T)). 
\end{align} 

As a first step, we show that $\norm{x(t)}_{\Theta} \le \beta(\norm{x_0}_{\Theta},t)$ at least for all $t \in [0,T)$. 
Indeed, in view of~(\ref{eq:norm-initially-less-than-r0,case(ii)}.a) and (\ref{eq:norm-initially-less-than-r0,case(ii)}.b) we have
\begin{align*}
\ol{\psi}(\norm{x(t)}_{\Theta}) \ge V(x(t)) > \ol{\psi}(\chi(\norm{u}_{\infty})) \ge \ol{\psi}\big(\chi(\norm{u(\cdot+t)}_{\infty})\big)
\qquad (t \in [0,T))
\end{align*}
and therefore we get by virtue of~\eqref{eq:lISS-L-fct,implicative} that
\begin{align}
\ol{\partial}_t^+ V(x(t)) &= \varlimsup_{\tau\to 0+} \frac{1}{\tau}\Big( V(x(t+\tau))-V(x(t)) \Big) \notag \\
&= \varlimsup_{\tau\to 0+} \frac{1}{\tau}\Big( V\big(S_{u(\cdot + t)}(\tau,0,x(t))\big)-V(x(t)) \Big) = \dot{V}_{u(\cdot + t)}(x(t)) \notag \\
&\le -\alpha(\norm{x(t)}_{\Theta}) \le -\big(\alpha \circ \ol{\psi}^{-1}\big)\big(V(x(t))\big)
\qquad (t \in [0,T)).
\end{align}
Consequently, by our choice of $\ol{\beta}$ we see that
\begin{align*}
V(x(t)) \le \ol{\beta}(V(x(0)),t) \qquad (t \in [0,T)).
\end{align*}
In view of~(\ref{eq:norm-initially-less-than-r0,case(ii)}.b) and our definition~\eqref{eq:def-beta-gamma} of $\beta$, the assertion of the first step is then clear. 
\smallskip

As a second step, we show that $T = t_0$. 
Indeed, assuming $T < t_0$, we would get by the first step and continuity that even $\norm{x(T)}_{\Theta} \le \beta(\norm{x_0}_{\Theta},T) \le \beta(r_{0x},0)$ and thus $\norm{x(T)}_{\Theta} < r_0$ by~\eqref{eq:def-r0x-and-r0u}. And from this, in turn, it would follow that $\norm{x(t)}_{\Theta} < r_0$ for all $t \in [T,T+\delta)$ with some $\delta > 0$. In conjunction with~(\ref{eq:norm-initially-less-than-r0,case(ii)}.b), this would yield a contradiction to the definition~\eqref{eq:def-T-case(ii)} of $T$!
\smallskip

As a third step, we show that $\norm{x(t)}_{\Theta} \le \gamma(\norm{u}_{\infty})$ for all $t \in [t_0,\infty)$. 
We can assume $t_0 < \infty$ 
because in the case $t_0 = \infty$ the assertion is empty. So, by 
the definition of $t_0$ it then follows that $x(t_0) \in M_u$ and therefore by virtue of~\eqref{eq:M_u-invariant} 
\begin{align*}
x(t) = S_u(t,0,x_0) = S_u(t,t_0,x(t_0)) \in M_u \qquad (t\in[t_0,\infty)).
\end{align*} 
In view of~\eqref{eq:M_u-contained-in-better-sets}, the assertion of the third step is then clear.
\smallskip

Combining now the first, second and third step, we finally obtain the desired estimates~\eqref{eq:part-(ii),1} and~\eqref{eq:part-(ii),2}, which clearly imply~\eqref{eq:lISS-estimate} in the case $x_0 \notin M_u$. 
\end{proof}

An inspection of the above proof shows that we actually proved a bit more than local input-to-state stability, namely we have: for every $r_0 > 0$ there exist $\beta \in \mathcal{KL}$ and $\gamma \in \mathcal{K}$ and $r_{0x}, r_{0u} > 0$ such that the estimate~\eqref{eq:lISS-estimate} holds true for all $\norm{x}_{\Theta} \le r_{0x}$ and $\norm{u}_{\infty} \le r_{0u}$. So, if by choosing $r_0$ large enough, we could also ensure that $r_{0x}$ and $r_{0u}$ with~\eqref{eq:def-r0x-and-r0u} can be chosen arbitrarily large, we would even have semi-global input-to-state stability. 
Yet, this is not so clear because the functions $\beta = \beta_{r_0}$ and $\gamma = \gamma_{r_0}$ from~\eqref{eq:def-r0x-and-r0u} which determine our choice of $r_{0x}$ and $r_{0u}$ depend on $r_0$ themselves (basically because $V = V_{r_0}$ and $\underline{\psi} = \underline{\psi}_{r_0}$ depend on $r_0$ as was pointed out after Lemma~\ref{lm:L-fct}). 
We therefore leave the question of semi-global input-to-state stability to future research. 

\section*{Acknowledgements}

S. Dashkovskiy and O. Kapustyan are partially supported by the German Research Foundation (DFG) and the State Fund for Fundamental Research of Ukraine (SFFRU) through the joint German-Ukrainian grant ``Stability and robustness of attractors of nonlinear infinite-dimensional systems with respect to disturbances'' (DA 767/12-1). 

\begin{small}

\end{small}

\end{document}